\documentclass[12pt,reqno]{amsart}
\usepackage{amsmath,amsfonts,graphicx,amscd,amssymb,epsf,amsthm,enumerate,parskip,mathrsfs,color,ltablex,blindtext,cancel,cases}  
 \usepackage[a4paper, margin=1in]{geometry}
\usepackage{times}
\usepackage{graphicx}
\usepackage{float}
\theoremstyle{definition}
\newtheorem{theorem}{Theorem}[section]

\newtheorem{lemma}[theorem]{Lemma}
\newtheorem{corollary}[theorem]{Corollary}
\newtheorem{definition}[theorem]{Definition}
\newtheorem{example}[theorem]{Example}

 \numberwithin{equation}{section}
\numberwithin{equation}{section}
\newcommand{\ZZ}{\mathbb{Z}}
\newcommand{\syst}{\operatorname{sys}}
\DeclareMathOperator{\Vol}{Vol}
\DeclareMathOperator{\ric}{Ric}

\usepackage{cite}
\usepackage[colorlinks,citecolor=blue]{hyperref}
\usepackage{amssymb}
\begin{document}
\makeatletter
\newcommand{\rmnum}[1]{\romannumeral #1}
\newcommand{\Rmnum}[1]{\expandafter\@slowromancap\romannumeral #1@}
\makeatother
\title{A 2-systolic inequality on non-rational compact Kähler surfaces with positive scalar curvature}
 \author{Zehao Sha}
\address{Institute for Mathematics and Fundamental Physics, 1899 Shilian South Road, 230088 Hefei, China}
\email{zhsha0251@gmail.com}
\begin{abstract}
In this note, we prove a 2-systolic inequality on non-rational compact positive scalar curvature Kähler surfaces admitting a nonconstant holomorphic map to a positive-genus compact Riemann surface. According to the classification of positive scalar curvature K\"ahler surfaces, any such surface must be a ruled surface fibred over a complex curve with positive genus.
\end{abstract}
%\subjclass[2020]{53A05}
%\keywords{Holomorphic maps; Scalar curvature; Homological 2-systole. }
\maketitle
%~~~~~~~~~~~~~~~~~~~~~~~~~~~~~~~~~~~~~~~~~~~~~~~~~~~~~~~~~~~~~~~~~~~~~~~~~~~~~~~~~~~~~~~~~~~~~~~~~~~~~~~~~~~~~~~~~~~~~~~~~~~~~~~~~~~~~~~~~~~~~~~~~~~~~~~~~~~~~~~
 
\section{Introduction}
Systolic geometry studies the relationship between the minimal size of non-trivial homological cycles in a Riemannian manifold and its global geometric properties. Given a closed Riemannian manifold \((M,h)\), the \emph{$k$-systole} is defined as
\[
\syst_k(M, h) := \inf \left\{ \Vol_g(Z) \mid Z \subset M \text{ embedded },\ [Z] \neq 0 \in H_k(M;\mathbb{Z}) \right\}.
\]

Consider a Riemannian manifold $(M,h)$ with positive scalar curvature (PSC for abbreviation), Schoen–Yau \cite{SchoenYau1979} confirmed that any area-minimizing surface in $M$ is homeomorphic to either $S^2$ or $\mathbb{RP}^2$. The $\mathbb{RP}^2$ case was studied in \cite{BBEN2010}, which proved the sharp upper bound of the area for an embedded area-minimizing projective plane,
\[
\operatorname{Area}_h(\Sigma)\cdot \min_M S_M \le 12\pi
\]
with equality if $M \simeq \mathbb{RP}^3$. In the $S^2$ case, Bray–Brendle–Neves \cite{BBN2010} established the following important rigidity result concerning the $\pi_2$-systole and and the minimum of the scalar curvature $\min S_M$ of a PSC $3$-manifold $(M,h)$:
\begin{theorem}[Bray-Brendle-Neves,\cite{BBN2010}]
    Let $(M^3, h)$ be a closed, orientable Riemannian 3-manifold with positive scalar curvature. Then the following inequality holds:
    \begin{equation} \label{BBN}
        \syst_{\pi_2}(M,h) \cdot \min_M S_M \leq 8\pi.
    \end{equation}
    Moreover, equality holds if and only if $M^3$ is isometrically covered by $S^2 \times S^1$ with the round metric on $S^2$, product with the flat metric on $S^1$.
\end{theorem}
The sharp bound in \eqref{BBN} was recently refined by Xu~\cite{Xu2025}, who proved that it is not only rigid but also exhibits a strict quantitative gap away from the model case. Further developments on the interplay between the $2$-systole and positive scalar curvature include \cite{ZhuPAMS2020,Richard2020,Orikasa2025}.

In \cite{stern2022scalar}, Stern introduced the following inequality for a non-constant $S^1$-valued harmonic map $u$ on a 3-manifold $(M,h)$ using the level set method,
\begin{equation}\label{Stern-ineq}
2 \pi \int_{\theta \in S^1} \chi\left(\Sigma_\theta\right)\geq \frac{1}{2} \int_{\theta \in S^1} \int_{\Sigma_\theta}\left(|d u|^{-2}|\operatorname{Hess}(u)|^2+S_M\right)
\end{equation}
and generalized the Bray–Brendle–Neves' systolic inequality for the homological 2-systole. 

Motivated by Stern’s approach, we adapt the level set method to compact Kähler surfaces fibred over a Riemann surface with positive genus, and obtain a Bray-Brendle-Neves type inequality for the homological $2$-systole. In particular, we have the following:
\begin{theorem}
    Let \((X,\omega)\) be a compact PSC K\"ahler surface admitting a non-constant holomorphic map \(f: X \to C\) to a compact Riemann surface \(C\) with genus $g(C)\ge 1$. Then, we have
    \begin{equation} \label{main_ineq}
    \min_X S_X\cdot\syst_2(X,\omega)  \leq  8 \pi.
    \end{equation}
Moreover, the equality holds if and only if $X$ is isometrically covered by $\mathbb{CP}^1\times \mathbb{C}$ equipped with the product of the standard Fubini–Study metric on \(\mathbb{CP}^1\) and a flat metric on \(\mathbb{C}\), so that $\syst_2(X,\omega)$ is achieved by $\mathbb{CP}^1$-fibre.
\end{theorem}
A direct consequence of the above result is:
\begin{corollary}
        Let \((X,\omega)\) be a compact PSC K\"ahler surface admitting a non-constant holomorphic map \(f: X \to C\) to a compact Riemann surface \(C\) with genus $g(S)\ge 2$. Then, we have
    \begin{equation} \label{main_ineq}
    \min_X S_X\cdot\syst_2(X,\omega)  < 8 \pi.
      \end{equation}
\end{corollary}

It is worth-noting that a compact K\"ahler surface $X$ admits a PSC metric if and only if $X$ is obtained from \(\mathbb{P}^2\) or \(\mathbb{P}(E)\) by a finite sequence of blow ups, where \(E\) is a rank 2 holomorphic vector bundle over a compact Riemann surface: For minimal compact K\"ahler surfaces (which are not the blow-up of other K\"ahler surfaces), Yau \cite{yau1974curvature} showed that the existence of a \emph{K\"ahler} PSC metric is equivalent to \(X \) is ruled or \(\mathbb{P}^2\) (see also \cite{LeBrun1995}). LeBrun conjectured that this statement is still valid after allowing the blowup. The remaining gap---whether blowing up preserves the \emph{sign} of the scalar curvature was settled recently by Brown, who proved that K\"ahler blow-ups preserve the sign of scalar curvature and completed the classification \cite[Thm~B]{Brown2024}.
%~~~~~~~~~~~~~~~~~~~~~~~~~~~~~~~~~~~~~~~~~~~~~~~~~~~~~~~~~~~~~~~~~~~~~~~~~~~~~~~~

\noindent \textbf{Acknowledgment.}~~The majority of this work was conducted while the author was a PhD student at Institut Fourier. The author thanks G. Brown for pointing out some useful suggestions. The author also extends his gratitude to his advisor, Professor Gérard Besson, for introducing him to the field of PSC manifolds.
%~~~~~~~~~~~~~~~~~~~~~~~~~~~~~~~~~~~~~~~~~~~~~~~~~~~~~~~~~~~~~~~~~~~~~~~~~~~~~~~~
\section{The K\"ahler surface fibred over a Riemann surface with positive genus}
Let $f: (X^{n},\omega) \rightarrow C$ be a non-constant holomorphic map from a compact K\"ahler manifold to a compact Riemann surface $C$ with genus $g \ge 1$. Thanks to the uniformization theorem, we can always choose a metric $\omega_0$ on $C$ with non-positive Gauss curvature. 

Taking $z \in C$ be any point, then we have $f^{-1}(z):=D_z$ is a Cartier divisor in $X$ with complex codimension $1$. Indeed, $D_z$ defines a line bundle $\mathcal{O}(D_z)$ with the first Chern class $c_1(\mathcal{O}(D_z))$ represented by $f^*\omega_0$ after normalization. In the subsequent part, we only consider the smooth part of $D_z$, and denote it by $D_z$ directly.

Recall the adjunction formula for a smooth divisor $D$, the canonical bundle $K_D$ satisfies
\begin{equation}
    K_D = \left(K_X\otimes \mathcal{O}(D)\right)|_D.
\end{equation}
Since $D$ is the fiber of a holomorphic map over $C$, the normal bundle $\mathcal{N}_D \cong \mathcal{O}(D)|_D$ is trivial. By taking the first Chern class of the adjunction formula, we obtain
\begin{equation*}
    c_1(D) = c_1(X)|_D,
\end{equation*}
which implies 
\begin{equation}
   \ric_D(\omega)=\ric_X(\omega)|_D+\sqrt{-1}\,\partial_D\bar\partial_D\phi,
\end{equation}
for a smooth real-valued function \(\phi\in C^\infty(D)\) thanks to the \(\partial\bar\partial\)-lemma. Denote $\nu$ by the unit normal vector field of $D$ of type $(1,0)$, we obtain the traced Gauss equation
\begin{align*}
   S_D(\omega)=S_X(\omega)-\ric_X(\omega)(\nu,\bar\nu)+\Delta_D\phi.
\end{align*}
Moreover, since $\nu=\nabla^{1,0}f / |\nabla^{1,0}f|$, we obtain
\begin{equation}\label{traced_Gauss}
\ric_X(\omega)\bigl(\nabla^{1,0}f,\nabla^{0,1}f\bigr)
=
|\nabla^{1,0}f|^2\bigl(S_X(\omega)-S_D(\omega)+\Delta_D\phi\bigr).
\end{equation}

Recall the Bochner formula for holomorphic maps, and the co-area formula:
\begin{lemma}[Bochner formula]
    Let $f : (X,\omega) \rightarrow (N,\tilde{\omega})$ be holomorphic, then
    \begin{equation}
        \Delta |\partial f|^2 = |\nabla \partial f|^2 + \langle \operatorname{Ric}(\omega) , f^* \tilde{\omega} \rangle- \operatorname{tr}^2_{\omega} \left(f^*\operatorname{Rm}(\tilde{\omega})\right).
    \end{equation}
where $\operatorname{Ric}(\omega)$ is the Ricci form of $X$, and $\operatorname{Rm}(\tilde{\omega})$ is the curvature form of $N$.
\end{lemma}
\begin{lemma}[Co-area formula]
    Let $(X^n,\omega)$ be a compact K\"ahler manifold and let $(C,\omega_0)$ be a compact Riemann surface such that
    $$
\int_{C} \omega_0 = \frac{1}{n},
    $$
    after normalization. Then for any $g\in C^\infty(X)$ and $z \in E$ regular value of $f$, we have
    \begin{equation}
        \int_X g~\omega^n = \int_{C} \left(\int_{f^{-1}(z)}\frac{g}{|\partial f|^2}~ \omega^{n-1}\right) \omega_0.
    \end{equation}
\end{lemma}

When $(N,\tilde{\omega})=(C,\omega_0)$, we have 
$$
\Delta |\partial f|^2 \ge \left|\nabla \partial f\right|^2 + \operatorname{Ric}_X(\omega)\left(\nabla^{1,0}f,\nabla^{0,1}f\right) ,
$$
with equality holds if $\operatorname{Rm}(\omega_0)\equiv 0$. Combining with the traced Gauss equation (\ref{traced_Gauss}), we obtain
\begin{align}\label{important}
\Delta |\partial f|^2
\;\ge\; |\nabla \partial f|^2
        + |\nabla^{1,0}f|^2\bigl(S_X(\omega)-S_D(\omega)+\Delta_D\phi\bigr).
\end{align}
Combining with the co-area formula and (\ref{important}), we can then see the following identity:
\begin{lemma}
Let \((X^{n},\omega)\) be a compact Kähler manifold and let \((C,\omega_0)\) be a
compact Riemann surface of genus \(g(C)\ge1\), endowed with a constant curvature
metric \(\omega_0\). Suppose that \(f\colon X \rightarrow C\) is a non-constant
holomorphic map, and let \(D=f^{-1}(z)\) denote a regular fibre. Then
\begin{align}\label{main_eq}
\int_{C} \left[\int_{D} \left(
\frac{\left|\nabla \partial f\right|^2}{|\partial f|^2}
+  S_X(\omega) -  S_{D}(\omega) \right)\omega^{n-1} \right]\omega_0 
\;\le\; 0.
\end{align}
Moreover, equality holds in \eqref{main_eq} if and only if \(g(B)=1\) and
\((C,\omega_0)\) is an elliptic curve with a flat metric.
\end{lemma}
\begin{proof}
Let $C=A \cup B$, where $A$ contains the set for all critical values of $f$. Then, Integrating \eqref{important} over $f^{-1}(B)$, we have
\begin{align}\label{a}
\int_{f^{-1}(B)}\Bigl(
\bigl|\nabla \partial f\bigr|^2
+ \bigl|\nabla^{1,0}f\bigr|^2\bigl(S_X(\omega) - S_{D}(\omega)+\Delta_D\phi\bigr)\Bigr)\omega^n
\;\le\; \int_{f^{-1}(B)} \bigl(\Delta |\partial f|^2\bigr) \omega^n.
\end{align}
Since \(X\) is compact and has no boundary,
\[
\int_X \bigl(\Delta |\partial f|^2\bigr)\omega^n=0.
\]
Hence, by choosing \(A\) with arbitrarily small measure and using Sard's theorem,
we may pass to the limit in \eqref{a} and obtain
\[
\int_X\Bigl(
\bigl|\nabla \partial f\bigr|^2
+ \bigl|\nabla^{1,0}f\bigr|^2\bigl(S_X(\omega) - S_{D}(\omega)+\Delta_D\phi\bigr)\Bigr)\omega^n
\;\le\; 0.
\]
We apply the co-area formula to the left-hand side of (\ref{a}), which gives the desired inequality. The equality statement follows from the discussion before \eqref{important}: equality in the Bochner inequality holds if and only if \(\operatorname{Rm}(\omega_0)\equiv0\), namely if and only if \(g(C)=1\) and \((C,\omega_0)\) is flat.
\end{proof}
%~~~~~~~~~~~~~~~~~~~~~~~~~~~~~~~~~~~~~~~~~~~~~~~~~~~~~~~~~~~~~~~~~~~~~~~~~~~~~~~~
\section{The 2-systole in K\"ahler surface}

This section is devoted to the study of the (homological) 2-systole in Kähler surfaces. We begin by recalling the fundamental definition of the $k$-systole in Riemannian geometry. Let $(M, h)$ be a closed Riemannian manifold of dimension $n \geq k$. The $k$-systole $\syst_k(M, h)$ is defined as the infimum of the volumes of all integral $k$-cycles representing nontrivial homology classes:
\[
\syst_k(M, h) := \inf \left\{ \Vol_h(Z) \mid Z \subset M\text{ embedded }, [Z]\neq 0 \in H_k(M; \ZZ) \right\}.
\]
In the context of Kähler geometry, additional structure enriches this concept. Let $(X,\omega)$ be a compact Kähler surface. The 2-systole is the least area among nonseparating real surfaces in $X$.
\begin{definition}[2-systole in Kähler surfaces]\label{def:2-systole-kahler}
For a compact Kähler surface $(X, \omega)$, the \emph{2-systole} can be defined by
\[
\syst_2(X, \omega) = \inf \left\{ \Vol_{\omega}(Z) \mid Z \subset X\text{ embedded }, [Z]\neq 0 \in H_2(X; \ZZ)\right\}.
\]
\end{definition}
The following result gives a Bray-Brendle-Neves type inequality \cite{BBN2010} for 2-systole for compact PSC K\"ahler surfaces over a Riemann surface with genus $g \ge 1$.
\begin{theorem}
    Let \((X,\omega)\) be a compact K\"ahler surface admitting a non-constant holomorphic map \(f: X \to C\) to a complex curve \(C\) with genus $g(C)\ge1$. Then, we have
    \begin{equation} \label{main_ineq}
    \min_X S_X\cdot\syst_2(X,\omega)  \leq  8 \pi.
    \end{equation}
Moreover, the equality holds if and only if $X$ is isometrically covered by $\mathbb{CP}^1\times \mathbb{C}$ equipped with the product of the standard Fubini–Study metric on \(\mathbb{CP}^1\) and a flat metric on \(\mathbb{C}\), so that $\syst_2(X,\omega)$ is achieved by $\mathbb{CP}^1$-fibre.
\end{theorem}
\begin{proof}
    It follows from (\ref{main_eq}),
    \begin{align*}
        \int_{C} \left[\int_{D_z} \left(\frac{\left|\nabla \partial f\right|^2}{|\partial f|^2}+S_X\right)\omega \right]\omega_0 \le  \int_{C} \left(\int_{D_z} S_{D_z}\cdot\omega \right)\omega_0.
    \end{align*}
    Then by the Gauss-Bonnet formula,
    \begin{align*}
       4\pi\int_{C} \chi(D_z)~\omega_0&= \int_{C} \left(\int_{D_z} S_{D_z}\cdot\omega \right)\omega_0  \\
       &\geq \int_{C} \left(\int_{D_z} S_{X}\cdot\omega \right)\omega_0 \\
       &\geq \min_X S_X \cdot\int_{C} \Vol_\omega(D_z)~\omega_0.
    \end{align*}
    Denote $N(z)$ by the number of the homological non-zero irreducible components of $D_z$, we then have
    $$
\Vol_\omega(D_z) \geq N(z) \cdot \syst_2(X,\omega).
    $$
    Meanwhile, since $D_z \simeq \mathbb{CP}^1$, we have
    $$
\chi(D_z) = 2 N(z).
    $$
    Thus, we have 
    \begin{align*}
        8 \pi \int_{C} N(z)~\omega_0 &= 4\pi\int_{C} \chi(D_z)~\omega_0 \\
        &\geq \min_X S_X \cdot\int_{C} \Vol_\omega(D_z)~\omega_0 \\
        &\geq \min_X S_X\cdot\syst_2(X,\omega) \int_{C}  N(z)~\omega_0 ,
    \end{align*}
    which gives the desired result. The equality holds in case $C$ admits a flat metric and $\nabla f$ is parallel along $D_z$, with each irreducible component of $D_z $ is $\mathbb{CP}^1$.
\end{proof}

We finally see a simple but interesting example:
    \begin{example}
Let $X=\mathbb{P}^1\times C$ be a compact complex surface, where $C$ is a compact Riemann surface of genus $g\ge2$. Equip $X$ with the product K\"ahler metric
\[
\omega=\omega_{\mathrm{FS}}\oplus\omega_C,
\]
where on $\mathbb{P}^1$ we take the Fubini--Study metric normalized by
\[
\Vol_{\omega_{\mathrm{FS}}}(\mathbb{P}^1)=\pi,\qquad S_{\mathbb{P}^1}=8,
\]
and on $C$ we choose a constant scalar curvature metric with
\[
S_C=-8+\varepsilon\quad\text{for some }\varepsilon\in(0,8).
\]
Then the product scalar curvature \(S_X=S_{\mathbb{P}^1}+S_C=8+(-8+\varepsilon)=\varepsilon\) is constant, hence $\min_X S_X=\varepsilon$ and $X$ has positive scalar curvature.

Next, compare the areas of the two basic complex curves:
\begin{itemize}
  \item For the $\mathbb{P}^1$-fiber $F=\mathbb{P}^1\times\{p\}$, calibration by $\omega$ gives
  \(
  \Vol_\omega(F)=\pi.
  \)
  \item For the $C$-fiber $C_p=\{q\}\times C$, Gauss--Bonnet formula yields
  \[
  \int_C K_C\,\omega=2\pi\chi(C)=2\pi(2-2g)=-4\pi(g-1).
  \]
  Then we obtain \[\Vol_\omega(C_p)=\frac{8\pi(g-1)}{\,8-\varepsilon\,}.\]
\end{itemize}
For $g\ge2$ and $\varepsilon\in(0,8)$ one has $\operatorname{Area}_\omega(C_p)>\pi$, so the $2$-systole is realized by the $\mathbb{P}^1$-fiber:
\[
\syst_2(X,\omega)=\min\big\{\operatorname{Area}_\omega(F),\operatorname{Area}_\omega(C_p)\big\}=\pi.
\]
Consequently,
\[
\min_X S_X\cdot\syst_2(X,\omega)=\varepsilon\cdot\pi<8\pi.
\]
In particular, this product is independent of the genus $g$, and it can be made arbitrarily close to $8\pi$ by letting $\varepsilon\uparrow 8$.
\end{example}

%~~~~~~~~~~~~~~~~~~~~~~~~~~~~~~~~~~~~~~~~~~~~~~~~~~~~~~~~~~~~~~~~~~~~~~~~~~~~~~~~
%~~~~~~~~~~~~~~~~~~~~~~~~~~~~~~~~~~~~~~~~~~~~~~~~~~~~~~~~~~~~~~~~~~~~~~~~~~~~~~~~
\bibliographystyle{amsalpha}
\bibliography{wpref}
\end{document}